\documentclass[10pt]{article}
 \font\smallit=cmti10
 \font\smalltt=cmtt10

\usepackage{amssymb,latexsym,amsmath,epsfig,amsthm} 

\usepackage{amsfonts,amsmath,amsthm,amssymb,amscd}

\makeatletter

\renewcommand\section{\@startsection {section}{1}{\z@}
 {-30pt \@plus -1ex \@minus -.2ex}
 {2.3ex \@plus.2ex}
 {\normalfont\normalsize\bfseries}}

\renewcommand\subsection{\@startsection{subsection}{2}{\z@}
 {-3.25ex\@plus -1ex \@minus -.2ex}
 {1.5ex \@plus .2ex}
 {\normalfont\normalsize\bfseries}}

\renewcommand{\@seccntformat}[1]{\csname the#1\endcsname. }

\makeatother

\newtheorem{theorem}{Theorem}
 \newtheorem{lemma}{Lemma}
 
 \newtheorem{proposition}{Proposition}
 \newtheorem{corollary}{Corollary}

\numberwithin{equation}{section}
\numberwithin{theorem}{section}
\numberwithin{proposition}{section}
\numberwithin{lemma}{section}
\numberwithin{corollary}{section}
\numberwithin{remark}{remark}

\begin{document}

\begin{center}
 \uppercase{\bf Mixed sums of triangular numbers and certain binary quadratic forms}
 \vskip 20pt
 {\bf Kazuhide Matsuda
}\\
 {\smallit Faculty of Fundamental Science, National Institute of Technology, Niihama College, Ehime, Japan.}\\
 {\tt matsuda@sci.niihama-nct.ac.jp}\\

 \end{center}
 \vskip 30pt
\centerline{\smallit Received: , Revised: , Accepted: , Published: } 
 \vskip 30pt

\centerline{\bf Abstract}

\noindent 
In this paper, we prove that for $d=3,\dots,8$, every natural number can be written as $t_x+t_y+3t_z+dt_w$, where $x$, $y$, $z$, and $w$ are nonnegative integers and $t_k=k(k+1)/2$ $(k=0,1,2,\ldots)$ is a triangular number. Furthermore, we study mixed sums of triangular numbers and certain binary quadratic forms.

\pagestyle{myheadings}
 \markright{\smalltt INTEGERS: 15 (2015)\hfill}
 \thispagestyle{empty}
 \baselineskip=12.875pt
 \vskip 30pt

\section{Introduction}
Let $\mathbb{N}=\{1,2,3,\ldots\}$, $\mathbb{N}_0=\{0, 1,2,3,\ldots\}$, and denote the set of squares by $\{x^{2}: x \in \mathbb{Z}\}.$ A {\em triangular} number is defined as $t_x=x(x+1)/2$, where $x\in\mathbb{N}_0$. Furthermore, for positive integers $j,k,n\in\mathbb{N},$ let $d_{j,k}(n)$ denote the number of positive divisors $d$ of $n$ such that $d\equiv j \,\,(\mathrm{mod} \,k).$
\par
A well-known result of Gauss states that every $n\in\mathbb{N}$ can be written as a sum of three triangular numbers; that is, $n=\Delta_1+\Delta_2+\Delta_3,$
where $\Delta_j \,(1\le j \le 3)$ is a triangular number. In 1862 and 1863, Liouville \cite{Liouville-1,Liouville-2} proved the following result:
\begin{theorem} $\mathrm{(Liouville)}$
\label{thm:Liouville}
{\it Let $a$, $b$, and $c$ be positive integers with $a\le b \le c.$ Then every $n\in\mathbb{N}_0$ can be written as $a t_{x}+b t_{y}+c t_{z}$ for $x,y,z\in\mathbb{N}_0$ if and only if $(a,b,c)$ is among the following vectors:
$$
(1,1,1), (1,1,2), (1,1,4), (1,1,5), (1,2,2), (1,2,3), (1,2,4).
$$
}
\end{theorem}
\par
In this paper, we prove that $n=\Delta_1+\Delta_2+3\Delta_3+d\Delta_4$, where $3\le d \le 8$ and $\Delta_j \,(1\le j \le 4)$ is a triangular number. To prove this conjecture, we use the results of Barrucand et al. \cite{Barrucand-Cooper-Hirschhorn} and Adiga et al. \cite{Adiga-Cooper-Han}, which were obtained using Ramanujan's theory of theta functions. For $q\in\mathbb{C}$ such that $|q|<1,$ we introduce
$$
\varphi(q)=\sum_{n\in\mathbb{Z}} q^{n^2}, \,\,
\psi(q)=\sum_{n=0}^{\infty} q^{n(n+1)/2}, \,\,
a(q)=\sum_{m,n\in\mathbb{Z}} q^{m^2+mn+n^2}.
$$
Note that Williams \cite{Williams} determined the number of representations of $n\in\mathbb{N}_0$ as $\Delta_1+\Delta_2+2(\Delta_3+\Delta_4)$
using an entirely arithmetic method.
\par
Our main theorems are as follows:

\begin{theorem}
\label{thm:(1,1,3,3)}
{\it
\quad
\begin{enumerate}\itemsep=0pt
\item[$(1)$]
For $d\in\mathbb{N}$ with $3\le d \le 8,$ every $n\in\mathbb{N}_0$ can be represented as $t_x+t_y+3t_z+dt_w$ for $x,y,z,w\in\mathbb{N}_0.$
\item[$(2)$]
Let $a$, $b$, $c$, and $d$ be positive integers with $a\le b \le c \le d.$ Then every $n\in\mathbb{N}_0$ can be written as $a t_{x}+b t_{y}+c t_{z}+dt_w$ for $x,y,z,w\in\mathbb{N}_0$ if and only if $(a,b,c,d)$ is among the following vectors:
\begin{align*}
&(1,1,1,d), (1,1,2,d), (1,1,4,d), (1,1,5,d), (1,2,2,d), (1,2,3,d), (1,2,4,d),  \\
&(1,1,3,3), (1,1,3,4), (1,1,3,5), (1,1,3,6), (1,1,3,7), (1,1,3,8).
\end{align*}
\end{enumerate}
}
\end{theorem}
\begin{theorem}
\label{thm:a-c}
{\it
For fixed positive integers $a$ and $c$, set
$$
g^{a}_c(x,y,z)=at_x+c(y^2+yz+z^2), \,\,\textrm{with} \,\,x\in\mathbb{N}_0, \,\,y,z\in\mathbb{Z}.
$$
\begin{enumerate}\itemsep=0pt
\item[$(1)$]
The form $g^{a}_c$ represents all $n\in\mathbb{N}_0$ if and only if $(a,c)=(1,1).$
\item[$(2)$]
If the form $g^{a}_c$ represents $n=1,2,4,8$, it represents all $n\in\mathbb{N}_0$.
\end{enumerate}
}
\end{theorem}
\begin{theorem}
\label{thm:a,b-c}
{\it
For fixed positive integers, $a$, $b$, and $c$ with $a\le b,$ set
$$
g^{a,b}_c(x,y,z,w)=at_x+bt_y+c(z^2+zw+w^2), \,\,\textrm{with} \,\,x,y\in\mathbb{N}_0, \,\,z,w\in\mathbb{Z}.
$$
\begin{enumerate}\itemsep=0pt
\item[$(1)$]
The form $g^{a,b}_c$ represents all $n\in\mathbb{N}_0$ if and only if
\begin{equation*}
(a,b,c)=
\begin{cases}
(1,b,1),  \,&(b\in\mathbb{N}),  \\
(2,b,1), \,&(b=2,3,4,5,6,7,8),  \\
(1,b,2), \,&(b=1,2,3,4),  \\
(1,2,3), \,&              \\
(1,b,4), \,&(b=1,2),  \\
(1,1,5). \,&
\end{cases}
\end{equation*}
\item[$(2)$]
If the form $g^{a,b}_c$ represents $n=1,2,4,5,8,$ it represents all $n\in\mathbb{N}_0$.
\end{enumerate}
}
\end{theorem}
\par
The remainder of this paper is organized as follows. In Section 2, we introduce notation that will be used throughout our paper. In Section 3, we prove Theorem \ref{thm:(1,1,3,3)} for $d=3,6,7,8$, and in Section 4, we prove Theorem \ref{thm:a-c}. In Section 5, we apply Theorem \ref{thm:(1,1,3,3)} to obtain the sufficiency of Theorem \ref{thm:a,b-c}(1). In Section 6, we prove Theorem \ref{thm:a,b-c}.

\subsubsection*{Remark 1}
For $d=4,5,$ Theorem \ref{thm:(1,1,3,3)}(1) follows from Theorem \ref{thm:Liouville} of Liouville.

\subsubsection*{Remark 2}
For fixed positive integers, $a$, $b$, and $c$, set
$$
f^{a,b,c}(x,y,z)=ax^2+by^2+cz^2, \,\,
f^{a}_c(x,y,z)=ax^2+c(y^2+yz+z^2),
$$
for $x,y,z\in\mathbb{Z}.$ In \cite[p. 104]{Dickson-2}, Dickson proved that there exist no positive integers $a$, $b$, and $c$ such that $f^{a,b,c}$ can represent all nonnegative integers. In \cite{Matsuda}, we showed that there exist no positive integers $a$ and $c$ such that $f^{a}_c$ represents all nonnegative integers.

\subsubsection*{Remark 3}
For fixed positive integers $a$, $b$, and $c$, set
$$
f^{a,b}_c(x,y,z)=ax^2+by^2+c(z^2+zw+w^2),
$$
for $x,y,z,w\in\mathbb{Z}.$ In \cite{Matsuda}, we determined $(a,b,c)$, where $a\le b$, such that $f^{a,b}_c$ represents all nonnegative integers. 

\section{Notation and preliminary results}

\subsection{Notation}
For fixed positive integers $a$, $b$, $c$, and $d$ and each $n\in\mathbb{N}_0$, we define
\begin{align*}
r_{a,b,c}(n)=&\sharp\{(x,y,z)\in\mathbb{Z}^3 \,|\, n=ax^2+by^2+cz^2\},  \\
r_{a,b,c,d}(n)=&\sharp\{(x,y,z,w)\in\mathbb{Z}^3 \,|\, n=ax^2+by^2+cz^2+dw^2\},  \\
t_{a,b,c}(n)=&\sharp\{(x,y,z)\in\mathbb{N}_0^3 \,|\, n=at_x+bt_y+ct_z\},  \\
t_{a,b,c,d}(n)=&\sharp\{(x,y,z,w)\in\mathbb{N}_0^4 \,|\, n=at_x+bt_y+ct_z+dt_w\},  \\
m_{a\text{-}b, c}(n)
=&\sharp
\left\{
(x,y,z)\in\mathbb{Z}\times\mathbb{N}_0^2 \, | \,n=a x^2+b t_y+c t_z
\right\},  \\
m_{a,b\text{-}c}(n)
=&\sharp
\left\{
(x,y,z)\in\mathbb{Z}^2\times\mathbb{N}_0 \, | \,n=a x^2+by^2+c t_z
\right\},  \\
A^a_c(n)=&\sharp
\left\{
(x,y,z)\in\mathbb{Z}^3 \,| \,
n=ax^2+c(y^2+yz+z^2)
\right\},  \\
A^{a,b}_c(n)=&\sharp
\left\{
(x,y,z,w)\in\mathbb{Z}^4 \,| \,
n=ax^2+by^2+c(z^2+zw+w^2)
\right\},  \\
B^a_c(n)=&\sharp
\left\{
(x,y,z)\in\mathbb{N}_0\times\mathbb{Z}^2 \,| \,
n=at_x+c(y^2+yz+z^2)
\right\},  \\
B^{a,b}_c(n)=&\sharp
\left\{
(x,y,z,w)\in\times\mathbb{N}_0^2\times\mathbb{Z}^2 \,| \,
n=at_x+bt_y+c(z^2+zw+w^2)
\right\}.
\end{align*}

\subsection{Ramanujan's theory of theta functions}
From Baruah, Cooper, and Hirschhorn \cite{Baruah-Cooper-Hirschhorn}, recall that
\begin{align}
\varphi(q)=& \,\,  \varphi(q^4)+2q\psi(q^8), \label{eqn:Cooper;(a)} \\
\varphi(q)^2=& \, \, \varphi(q^2)^2+4q\psi(q^4)^2,  \label{eqn:Cooper;(b)}  \\
\varphi(q)\psi(q^2)=& \,\,  \psi(q)^2,    \label{eqn:Cooper;(c)}  \\
\varphi(q)\varphi(q^3)=& \,\,  a(q^4)+2q\psi(q^2)\psi(q^6),  \label{eqn:Cooper;(d)}  \\
a(q)=& \, \, \varphi(q)\varphi(q^3)+4q\psi(q^2)\psi(q^6),  \label{eqn:Cooper;(e)}  \\
a(q)=& \, \, a(q^4)+6q\psi(q^2)\psi(q^6).   \label{eqn:Cooper;(f)}
\end{align}

\section{Proof of Theorem \ref{thm:(1,1,3,3)} }

\subsection{Preliminary results}
Using Ramanujan's theory of theta functions, Adiga, Cooper, and Han \cite{Adiga-Cooper-Han} proved the following theorem:
\begin{theorem}
\label{thm:relation}
{\it Let $n\in\mathbb{N}_0.$ Then,
$$
r_{1,1,3}(8n+5)=16 t_{1,1,3}(n).
$$
}
\end{theorem}

\begin{proof}
For a detailed proof, see Barrucand et al. \cite{Barrucand-Cooper-Hirschhorn} and Adiga et al. \cite{Adiga-Cooper-Han}.
\end{proof}

Dickson \cite[p. 112-113]{Dickson-2} proved the following result:
\begin{theorem}
\label{thm:three-or four-squares}
{\it
\quad
\begin{enumerate}\itemsep=0pt
\item[$(1)$]
A nonnegative integer $n\in\mathbb{N}_0$ can be written as $x^2+y^2+3z^2$ for $x,y,z\in\mathbb{Z}$ if and only if $n\neq 9^k(9l+6)$ for $k,l\in\mathbb{N}_0.$
\item[$(2)$]
Every $n\in\mathbb{N}_0$ can be written as $x^2+y^2+3z^2+3w^2$ for $x,y,z,w\in\mathbb{Z}.$
\end{enumerate}
}
\end{theorem}

Using Theorems \ref{thm:relation} and \ref{thm:three-or four-squares}, the following theorem is obtained:

\begin{theorem}
\label{thm:(1,1,3)}
{\it
A nonnegative integer $n\in\mathbb{N}_0$ can be written as $t_x+t_y+3t_z$ for $x,y,z\in\mathbb{N}_0$ if and only if $n$ satisfies one of the following conditions:
\begin{enumerate}\itemsep=0pt
\item[$(1)$] $n\not\equiv 5, 8 \,(\mathrm{mod} \,9),$
\item[$(2)$] $n\equiv 5 \,(\mathrm{mod}\,9)$ and $8n+5=9^k(8N+5)$ and $N\not\equiv 5, 8\, (\mathrm{mod} \,9)$ for $k\in\mathbb{N}$ and $N\in\mathbb{N}_0.$
\end{enumerate}
}
\end{theorem}

\begin{proof}
From Theorems \ref{thm:relation} and \ref{thm:three-or four-squares}, note that $n=t_x+t_y+3t_z$ for $x,y,z\in\mathbb{N}_0$ if and only if $8n+5\neq 9^k(9l+6)$ where $k,l\in\mathbb{N}_0.$ Moreover, $8n+5\equiv 0\,(\mathrm{mod}\,9)$ if and only if $n\equiv 5\,(\mathrm{mod}\, 9).$
\par
First, consider the case when $n\not\equiv 5 \,(\mathrm{mod} \,9).$ From Theorem \ref{thm:three-or four-squares}, condition (1) is obtained because $8n+5\equiv 6 \,(\mathrm{mod}\, 9)$ if and only if $n\equiv 8\,(\mathrm{mod}\, 9).$
\par
Next, we consider the case when $n\equiv 5\,(\mathrm{mod}\, 9)$ and $n=9N+5$ for $N\in\mathbb{N}_0,$ which implies $8n+5=9(8N+5).$ Therefore, we set $8n+5=9^k(8N^{\prime}+5)$  for $k\in\mathbb{N}, N^{\prime}\in\mathbb{N}_0,$ where $N^{\prime}\not\equiv 5 \,(\mathrm{mod}\,9).$
From Theorem \ref{thm:three-or four-squares} and the discussion in the first paragraph, we obtain condition (2).
\end{proof}

\subsection{Proof of Theorem \ref{thm:(1,1,3,3)} (1) for $d=3$ }

\subsubsection{Preliminary results}
Before proving Theorem \ref{thm:(1,1,3,3)} (1) for $d=3$, we first obtain a useful number theoretic property for $t_{1,1,3,3}(n).$

\begin{theorem}
\label{thm:property}
{\it
Let $k\in\mathbb{N}$ and $n, N\in\mathbb{N}_0.$ Then,
$$
t_{1,1,3,3}(2^k N+(2^k-1))=2^k t_{1,1,3,3}(N),
$$
which implies
$
t_{1,1,3,3}(n) \equiv 0 (\bmod \, 2^k)
$
if $n\equiv -1 \,(\bmod \, 2^k).$
}
\end{theorem}

\begin{proof}
Multiplying both sides of (\ref{eqn:Cooper;(f)}) by $\psi(q^2)\psi(q^6)$ yields
$$
\psi(q^2)\psi(q^6)a(q)=\psi(q^2)\psi(q^6)a(q^4)+6q\psi(q^2)^2 \psi(q^6)^2,
$$
which implies
\begin{equation}
\label{eqn:relation-tri-1133-a}
\sum_{n=0}^{\infty} B^{2,6}_1 (n) q^n
=
\sum_{N=0}^{\infty} B^{1,3}_2(N) q^{2N}+6\sum_{N=0}^{\infty} t_{1,1,3,3} (N) q^{2N+1}.
\end{equation}
Multiplying both sides of (\ref{eqn:Cooper;(e)}) by $\psi(q^2)\psi(q^6)$ yields
$$
\psi(q^2)\psi(q^6)a(q)=\varphi(q)\varphi(q^3)\psi(q^2)\psi(q^6)+4q\psi(q^2)^2\psi(q^6)^2.
$$
From (\ref{eqn:Cooper;(c)}), we obtain
$$
\psi(q^2)\psi(q^6)a(q)=\psi(q)^2 \psi(q^3)^2+4q\psi(q^2)^2 \psi(q^6)^2,
$$
which implies
\begin{equation}
\label{eqn:relation-tri-1133-b}
\sum_{n=0}^{\infty} B^{2,6}_1(n) q^n
=
\sum_{n=0}^{\infty} t_{1,1,3,3} (n) q^n+4\sum_{N=0}^{\infty} t_{1,1,3,3} (N) q^{2N+1}.
\end{equation}
\par
From (\ref{eqn:relation-tri-1133-a}) and (\ref{eqn:relation-tri-1133-b}), note that
$$
B^{2,6}_1 (2N+1)=6t_{1,1,3,3} (N)=t_{1,1,3,3} (2N+1)+4t_{1,1,3,3}(N),
$$
which implies 
\begin{equation}
\label{eqn:relation-tri-1133(3)}
t_{1,1,3,3} (2N+1)=2t_{1,1,3,3} (N).
\end{equation}
\par
Using induction, we complete the proof of the theorem. Clearly, for $k=1,$ the theorem holds. Now, suppose the theorem holds for $k$. In this case,
\begin{align*}
t_{1,1,3,3}(2^{k+1} N+(2^{k+1}-1))
=& \, t_{1,1,3,3}\left( 2^k\cdot(2N+1)+(2^k-1)  \right)  \\
=& \, 2^k t_{1,1,3,3}(2N+1) \\
=& \, 2^{k+1} t_{1,1,3,3}(N).
\end{align*}
\end{proof}

From Baruah, Cooper, and Hirschhorn \cite{Baruah-Cooper-Hirschhorn}, recall the following result:

\begin{theorem}$\mathrm{(Baruah, Cooper \,\,and \,\, Hirschhorn)}$
\label{thm:Baruah-Cooper-Hirschhorn}
{\it
For every $n\in\mathbb{N}_0,$
\begin{equation*}
t_{1,1,3,3}(n)
=
\begin{cases}
\displaystyle
\frac14 r_{1,1,3,3}(n+1) &{\it if} \,n \,{\it is \,\,even,} \\

\displaystyle
\frac18 \{  r_{1,1,3,3}(2n+2)-  r_{1,1,3,3}(n+1) \}  &{\it if} \,n \,{\it is \,\,odd.}
\end{cases}
\end{equation*}
}
\end{theorem}

\subsubsection{Proof of Theorem \ref{thm:(1,1,3,3)} (1) for $d=3$ }

\begin{proof}
From Theorem \ref{thm:three-or four-squares}, first note that for every $N\in\mathbb{N}_0,$ $r_{1,1,3,3}(N)>0.$ If $n\in\mathbb{N}_0$ is even and $n=2N$ for $N\in\mathbb{N}_0,$ by Theorem \ref{thm:Baruah-Cooper-Hirschhorn}, we obtain
$$
t_{1,1,3,3}(n)=t_{1,1,3,3}(2N)=\frac14 r_{1,1,3,3}(2N+1)>0.
$$
\par
Suppose that $n\in\mathbb{N}_0$ is odd and $n+1=2^k(2N+1)$ for $k\in\mathbb{N}$ and $N\in\mathbb{N}_0.$ Then,
$$
n=2^k\cdot (2N)+(2^k-1).
$$
From Theorems \ref{thm:property} and \ref{thm:Baruah-Cooper-Hirschhorn}, it follows that
$$
t_{1,1,3,3}(n)=t_{1,1,3,3} \left(2^k\cdot (2N)+(2^k-1) \right)=2^k t_{1,1,3,3} (2N)=\frac{2^k}{4} r_{1,1,3,3}(2N+1)>0.
$$
\end{proof}

From the proof of Theorem \ref{thm:(1,1,3,3)} (1) for $d=3$, we can improve Theorem \ref{thm:Baruah-Cooper-Hirschhorn} of Baruah, Cooper, and Hirschhorn \cite{Baruah-Cooper-Hirschhorn} when $n$ is an odd number.

\begin{corollary}
\label{coro:Cooper-improvement}
{\it
\quad
\begin{enumerate}\itemsep=0pt
\item[$(1)$]
Suppose that $n\in\mathbb{N}_0$ is even and $n=2N$ for $N\in\mathbb{N}_0.$
Then,
$$
t_{1,1,3,3}(n)=\frac14 r_{1,1,3,3}(2N+1).
$$
\item[$(2)$]
Suppose that $n\in\mathbb{N}_0$ is odd and $n+1=2^k(2N+1)$ for $k\in\mathbb{N}$ and $N\in\mathbb{N}_0.$ Then,
$$
t_{1,1,3,3}(n)=\frac{2^k}{4} r_{1,1,3,3}(2N+1).
$$
\end{enumerate}
}
\end{corollary}

\begin{corollary}
{\it
For $k\in\mathbb{N}$ and $N\in\mathbb{N}_0,$ 
$$
r_{1,1,3,3}(2^k (2N+1))
=r_{1,1,3,3}(2(2N+1))+4(2^{k-1}-1)r_{1,1,3,3}(2N+1).
$$
}
\end{corollary}

\begin{proof}
For $k,n\in\mathbb{N}$ and $N\in\mathbb{N}_0,$ set
$$
n+1=2^k(2N+1),
$$
which implies $n=2^k\cdot(2N)+2^k-1.$
\par
By Theorem \ref{thm:Baruah-Cooper-Hirschhorn} and Corollary \ref{coro:Cooper-improvement} (2), it follows that
\begin{align*}
t_{1,1,3,3}(n)=&\frac18 \{  r_{1,1,3,3}(2n+2)-  r_{1,1,3,3}(n+1) \} \\
               =&\frac18 \{ r_{1,1,3,3}(2^{k+1}(2N+1))-r_{1,1,3,3}(2^k(2N+1)) \}  \\
               =&\frac{2^k}{4} r_{1,1,3,3}(2N+1),
\end{align*}
which implies 
$$
r_{1,1,3,3}(2^{k+1}(2N+1))=r_{1,1,3,3}(2^k(2N+1))+2^{k+1} r_{1,1,3,3}(2N+1).
$$
Solving this recurrence relation with respect to $k$ completes the proof.
\end{proof}

\subsection{Proof of Theorem \ref{thm:(1,1,3,3)} (1) for $d=6$}

\begin{proof}
From Theorem \ref{thm:(1,1,3)}, we only need to prove that $n\in\mathbb{N}_0$ can be written as $t_x+t_y+3t_z+6t_w$ for $x,y,z,w\in\mathbb{N}_0$ if $n\equiv 5 \,\mathrm{or} \,8 (\bmod \, 9).$
\par
When $n\equiv 8 (\bmod \, 9),$ taking $w=1$ yields
$$
n-6t_1\equiv 8-6\cdot 1\equiv 2 (\bmod \, 9),
$$
which implies from Theorem \ref{thm:(1,1,3)} that $n-6t_1$ can be expressed as $t_x+t_y+3t_z$ for $x,y,z\in\mathbb{N}_0.$
\par
When $n \equiv 5 (\bmod \, 9),$ by Theorem \ref{thm:(1,1,3)}, we assume that 
$$
8n+5=9^k (8N+5), \,\, k, N\in\mathbb{N}, \,\,N \equiv 8 (\bmod \, 9).
$$
Taking $w=2,$ note that  $n-6t_2=n-6\cdot 3\equiv 5 (\bmod \, 9)$ and
\begin{align*}
8(n-6t_2)+5=&8n+5-8\cdot 2 \cdot 9  \\
               =&9^k (8N+5)-8\cdot 2 \cdot 9  \\
               =&9\left\{ 9^{k-1}(8N+5)-8\cdot 2   \right\}  \\
               =&9\left\{  8\left( 9^{k-1} N+\frac{9^{k-1}\cdot 5-5}{8}-2 \right)+5 \right\}.
\end{align*}
When $k=1,$ 
$$
9^{k-1} N+\frac{9^{k-1}\cdot 5-5}{8}-2= N-2\equiv 8-2\equiv 6 (\bmod \, 9),
$$
which implies $n-6t_2$ can be expressed as $t_x+t_y+3t_z$ for $x,y,z\in\mathbb{N}_0.$
\par
When $k\ge 2,$ we obtain
$$
9^{k-1} N+\frac{9^{k-1}\cdot 5-5}{8}-2 \equiv -8\cdot 5-2\equiv 3 (\bmod \, 9),
$$
which implies $n-6t_2$ can be expressed as $t_x+t_y+3t_z$ for $x,y,z\in\mathbb{N}_0.$
\end{proof}

\subsection{Proof of Theorem \ref{thm:(1,1,3,3)} (1) for $d=7,8$}

\begin{proof}
By Theorem \ref{thm:(1,1,3)}, we are reduced to proving that $n\in\mathbb{N}_0$ can be written as $t_x+t_y+3t_z+dt_w$ for $x,y,z,w\in\mathbb{N}_0$ if $n\equiv 5 \,\mathrm{or} \,8 (\bmod \, 9).$
\par
Taking $w=1,$ we have
$$
n-d \cdot t_1\not\equiv 5,8 (\bmod \, 9),
$$
which implies $n-d \cdot t_1$ can be expressed as $t_x+t_y+3t_z$ for $x,y,z\in\mathbb{N}_0.$
\end{proof}

\subsection{Proof of Theorem \ref{thm:(1,1,3,3)} (2)}

\begin{proof}
Suppose that every $n\in\mathbb{N}_0$ can be expressed as $at_x+bt_y+ct_z+dt_w$ for $x,y,z,w\in\mathbb{N}_0.$ Taking $n=1,2$ yields
$$
(a,b)=(1,1),  (1,2).
$$
\par
First, we consider the case when $(a,b)=(1,2).$ Choosing $n=4$ implies 
$$
(a,b,c,d)=(1,2,2,d), (1,2,3,d), (1,2,4,d).
$$
\par
Next, we consider the case when $(a,b)=(1,1).$ If $n=5$, then $c=1,2,3,4,5$, which implies 
$$
(a,b,c,d)=(1,1,1,d), (1,1,2,d), (1,1,3,d), (1,1,4,d), (1,1,5,d).
$$
When $(a,b,c,d)=(1,1,3,d),$ taking $n=8$ yields
$$
(a,b,c,d)=(1,1,3,3), (1,1,3,4), (1,1,3,5), (1,1,3,6), (1,1,3,7), (1,1,3,8).
$$
Thus, the necessary conditions are obtained; sufficiency follows from Theorems \ref{thm:Liouville} and \ref{thm:(1,1,3,3)} (1).
\end{proof}

\section{Proof of Theorem \ref{thm:a-c}}

\subsection{Preliminary results}
First, note that for each positive integer $n\in\mathbb{N},$
\begin{equation}
\label{eqn:d=-3}
\sharp
\{
(x,y)\in\mathbb{Z}^2 |
n=x^2+xy+y^2
\}
=
6
(d_{1,3}(n)-d_{2,3}(n)).
\end{equation}
To prove this formula, we refer to Berndt \cite[p. 79]{Berndt}. Formula (\ref{eqn:d=-3}) implies $n=2,5,6,8$ cannot be expressed as $x^2+xy+y^2$ for $x,y,z\in\mathbb{Z}.$
\par
Next, note that every $n\in\mathbb{N}_0$ can be expressed as $x^2+3y^2+t_z$ for $(x,y,z)\in\mathbb{Z}^2\times\mathbb{N}_0,$ which was proven by Guo, Pan, and Sun \cite{Sun}.
\par
Finally, consider the following formula from Baruah, Cooper, and Hirschhorn \cite{Baruah-Cooper-Hirschhorn}:
\begin{equation}
\label{eqn:Cooper1}
a(q)=
\varphi(q)\varphi(q^3)+
4q\psi(q^2)\psi(q^6),
\end{equation}
where
$$
\varphi(q)=\sum_{n\in\mathbb{Z}} q^{n^2}, \,\,
\psi(q)=\sum_{n=0}^{\infty} q^{n(n+1)/2}, \,\,
a(q)=\sum_{m,n\in\mathbb{Z}} q^{m^2+mn+n^2}.
$$

\subsection{Proof of Theorem \ref{thm:a-c}}
\begin{proof}

If $n=1$, then $a=1$ or $c=1.$ Taking $n=2,$ we obtain
$$
(a,c)=(1,1), (1,2), (2,1).
$$
Choosing $n=4$ implies $(a,c)\neq (1,2)$; choosing $n=8$ implies $(a,c)\neq (2,1).$
\par
From (\ref{eqn:Cooper1}), we have
$$
\psi(q) a(q)=
\varphi(q)\varphi(q^3)\psi(q)+
4q \psi(q)\psi(q^2)\psi(q^6),
$$
which implies
$$
\sum_{n=0}^{\infty} B^1_1(n) q^n=
\sum_{n=0}^{\infty} m_{1,3\textrm{-}1} (n) q^n
+
4
\sum_{N=0}^{\infty}
t_{1,2,6}(N) q^{N+1}.
$$
\par
From a result of Guo, Pan, and Sun \cite{Sun}, it follows that $m_{1,3\textrm{-}1} (n)>0.$ Therefore, $B^1_1(n)>0,$ which means that every $n\in\mathbb{N}_0$ can be expressed as $t_x+(y^2+yz+z^2)$ for $(x,y,z)\in\mathbb{N}_0\times\mathbb{Z}^2$; thus, Theorem \ref{thm:a-c} (1) holds.
\par
Theorem \ref{thm:a-c} (2) follows from the above discussion.
\end{proof}

\section{Applications of Theorem \ref{thm:(1,1,3,3)}}

Theorem \ref{thm:(1,1,3,3)} is used to prove the following theorems; in particular, to prove Theorem \ref{thm:2-triangular}, we use the fact that every $n\in\mathbb{N}_0$ can be written as $t_x+t_y+3(t_z+t_w)$ for $x,y,z,w\in\mathbb{N}_0.$

\begin{theorem}
\label{thm:2-triangular}
{\it Let $b\in\mathbb{N}$ and $2\le b \le 8.$ Then every $n\in\mathbb{N}_0$ can be expressed as $2t_x+bt_y +(z^2+zw+w^2)$ for $(x,y,z,w)\in\mathbb{N}_0^2\times\mathbb{Z}^2.$}
\end{theorem}

\begin{theorem}
\label{thm:2-triangular(2)}
{\it Let $b\in\mathbb{N}$ and $1\le b \le 4.$ Then every $n\in\mathbb{N}_0$ can be expressed as $t_x+bt_y +2(z^2+zw+w^2)$ for $(x,y,z,w)\in\mathbb{N}_0^2\times\mathbb{Z}^2.$}
\end{theorem}

\begin{theorem}
\label{thm:2-triangular(3)}
{\it Every $n\in\mathbb{N}_0$ can be expressed as $t_x+2t_y +3(z^2+zw+w^2)$ for $(x,y,z,w)\in\mathbb{N}_0^2\times\mathbb{Z}^2.$}
\end{theorem}

\begin{theorem}
\label{thm:2-triangular(4)}
{\it Let $b=1,2.$ Then every $n\in\mathbb{N}_0$ can be expressed as $t_x+bt_y +4(z^2+zw+w^2)$ for $(x,y,z,w)\in\mathbb{N}_0^2\times\mathbb{Z}^2.$}
\end{theorem}

\begin{theorem}
\label{thm:2-triangular(5)}
{\it Every $n\in\mathbb{N}_0$ can be expressed as $t_x+t_y +5(z^2+zw+w^2)$ for $(x,y,z,w)\in\mathbb{N}_0^2\times\mathbb{Z}^2.$}
\end{theorem}

\subsection{Proof of Theorem \ref{thm:2-triangular} }

\subsubsection{Preliminary results}
Consider the following result by Dickson \cite[p. 112-113]{Dickson-2}:

\begin{lemma}
\label{lem-1,4,12}
{\it
A nonnegative integer $n\in\mathbb{N}_0$ can be written as $x^2+4y^2+12z^2$ for $x,y,z\in\mathbb{Z}$ if and only if $n\neq 4l+2, \, 4l+3, \, 9^k(9l+6)$, where $k,l\in\mathbb{N}_0.$
}
\end{lemma}

Lemma \ref{lem-1,4,12} gives rise to the following proposition:

\begin{proposition}
\label{prop:A-a=1,c=4}
{\it
A nonnegative integer $n\in\mathbb{N}_0$ can be written as $x^2+4(y^2+yz+z^2)$ for $x,y,z\in\mathbb{Z}$ if and only if $n\neq 4l+2, \, 4l+3, \, 9^k(9l+6)$, where $k,l\in\mathbb{N}_0.$
}
\end{proposition}

\begin{proof}
Replacing $q$ by $q^4$ in (\ref{eqn:Cooper;(f)}) and (\ref{eqn:Cooper;(e)}) yields
\begin{align*}
a(q^4)=&a(q^{16})+6q^4\psi (q^8)\psi (q^{24}),   \\
\varphi(q^4)\varphi(q^{12})=&a(q^{16})+2q^4\psi (q^8)\psi (q^{24}).
\end{align*}
Multiplying both sides of these equations by $\varphi(q)$ results in 
\begin{align}
\varphi(q)a(q^4)=&\varphi(q)a(q^{16})+6q^4\varphi(q)\psi (q^8)\psi (q^{24}),  \\
\varphi(q)\varphi(q^4)\varphi(q^{12})=&\varphi(q)a(q^{16})+2q^4\varphi(q)\psi (q^8)\psi (q^{24}),
\end{align}
which implies
\begin{align*}
\sum_{n=0}^{\infty} A^1_4(n) q^n=&\sum_{n=0}^{\infty} A^1_{16}(n) q^n +6q^4 \sum_{N=0}^{\infty} m_{1\text{-}8,24}(N)q^N,  \\
\sum_{n=0}^{\infty} r_{1,4,12}(n) q^n=&\sum_{n=0}^{\infty} A^1_{16}(n) q^n +2q^4 \sum_{N=0}^{\infty} m_{1\text{-}8,24}(N)q^N.
\end{align*}
Therefore, it follows that $n\neq 4l+2, \, 4l+3, \, 9^k(9l+6)$ , where $k,l\in\mathbb{N}_0$, if and only if $ r_{1,4,12}(n)>0,$ if and only if $A^1_{16}(n)>0$ or $m_{1\text{-}8,24}(n-4)>0,$ and if and only if $A^1_4(n)>0,$ which proves the proposition.
\end{proof}

Using Proposition \ref{prop:A-a=1,c=4}, we can prove the following proposition:

\begin{proposition}
\label{prop:B-a=2,c=1}
{\it A nonnegative integer $N\in\mathbb{N}_0$ can be written as $2t_x+(y^2+yz+z^2)$ for $(x,y,z)\in\mathbb{N}_0\times\mathbb{Z}^2$ if and only if either of the following occurs:
\begin{enumerate}\itemsep=0pt
\item[$(1)$] $N\not\equiv 2, 8 \,(\bmod \,9),$
\item[$(2)$] $N\equiv 2 \,(\bmod \,9)$ and $4N+1=9^k(4 N^{\prime}+1), \, N^{\prime} \not\equiv 2,8 (\bmod \, 9), \,k\in\mathbb{N}, N^{\prime}\in\mathbb{N}_0.$
\end{enumerate}
}
\end{proposition}

\begin{proof}
Multiplying both sides of (\ref{eqn:Cooper;(a)}) by $a(q^4)$ yields
$$
\varphi(q)a(q^4)=\varphi(q^4)a(q^4)+2q\psi(q^8)a(q^4),
$$
which implies 
$$
\sum_{n=0}^{\infty} A^1_4(n) q^n=\sum_{N=0}^{\infty} A^1_1(N) q^{4N}+2\sum_{N=0}^{\infty} B^2_1(N)q^{4N+1}.
$$
We then obtain
$$
B^2_1(N)>0 \,\mathrm{if \,and \, only \, if}\,  A^1_4(4N+1)>0.
$$
Therefore, the proposition follows from Proposition \ref{prop:A-a=1,c=4} and the facts that
$$
4N+1\equiv 0 \,(\bmod \, 9) \,\mathrm{if \,and \, only \, if}\, N\equiv 2 \,(\bmod \, 9),
$$
and
$$
4N+1\equiv 6 \,(\bmod \, 9) \,\mathrm{if \,and \, only \, if}\, N\equiv 8 \,(\bmod \, 9).
$$
\end{proof}

\subsubsection{Proof of Theorem \ref{thm:2-triangular} for $b\neq 3, 6$}

\begin{proof}
Because of Proposition \ref{prop:B-a=2,c=1}, we are reduced to proving that $N\in\mathbb{N}_0$ can be written as $2t_x+bt_y+(z^2+zw+w^2)$ for $(x,y,z,w)\in\mathbb{N}_0^2\times\mathbb{Z}^2$ if $N\equiv 2$ or $8 \,(\bmod \,9).$
\par
Suppose $b=2.$ When $N\equiv 2 (\bmod \, 9),$ taking $y=2$ results in $N-2\cdot 3\equiv 5 (\bmod \, 9),$ which implies $B^2_1(N-2\cdot t_2)>0.$ When $N\equiv 8 (\bmod \, 9),$ taking $y=1$ results in $N-2\cdot 1\equiv 6 (\bmod \, 9),$ which implies $B^2_1(N-2\cdot t_1)>0.$
\par
Next, suppose $b=4,5,7.$ Taking $y=1$ yields $N-b\cdot 1\not\equiv 2,8 (\bmod \, 9),$ which implies $B^2_1(N-b\cdot t_1)>0.$
\par
Suppose $b=8.$ When $N\equiv 2 (\bmod \, 9),$ taking $y=1$ results in $N-8\cdot 1\equiv 3 (\bmod \, 9),$ which implies $B^2_1(N-8\cdot t_1)>0.$ When $N=8, 17,26,35,44,$ we take
$$
(x,y,z,w)=(0,1,0,0), (0,1,3,0), (1,1,4,0), (0,1,3,3), (0,1,6,0).
$$
When $N\equiv 8 (\bmod \, 9)$ and $N>44,$ taking $y=3$ yields $N-8\cdot6\equiv 5 (\bmod \, 9),$ which implies $B^2_1(N-8\cdot t_3)>0.$
\end{proof}

\subsubsection{Proof of Theorem \ref{thm:2-triangular} for $b=3$}

\begin{proof}
By Proposition \ref{prop:B-a=2,c=1}, it suffices to show that $N\in\mathbb{N}_0$ can be written as $2t_x+3t_y+(z^2+zw+w^2)$ for $(x,y,z,w)\in\mathbb{N}_0^2\times\mathbb{Z}^2$ if $N\equiv 2 \,\mathrm{or} \,8 \,(\bmod \,9).$
\par
First, consider the case when $N\equiv 2 (\bmod \, 9).$ By Proposition \ref{prop:B-a=2,c=1}, we assume that $4N+1=9^k(4N^{\prime}+1),$ where $k\in\mathbb{N},\, N^{\prime}\in\mathbb{N}_0$, and $N^{\prime}\equiv 8 (\bmod \, 9).$ Taking $y=2$ results in
$$
N-3\cdot t_2=N-3\cdot 3\equiv 2 (\bmod \, 9).
$$
We then obtain
\begin{align*}
4(N-3\cdot t_2)+1=&4(N-3\cdot 3)+1  \\
                        =&4N+1-36  \\
                        =&9^k(4N^{\prime}+1)-36  \\
                        =&9\left\{ 4\left( 9^{k-1}N^{\prime}+\frac{9^{k-1}-1}{4}-1   \right) +1 \right\}.
\end{align*}
\par
When $k=1,$ 
$$
9^{k-1}N^{\prime}+\frac{9^{k-1}-1}{4}-1\equiv 8-1\equiv 7 (\bmod \, 9),
$$
which implies $B^2_1(N-3\cdot t_2)>0. $
\par
When $k\ge 2,$ we obtain
$$
9^{k-1}N^{\prime}+\frac{9^{k-1}-1}{4}-1\equiv 2-1\equiv 1 (\bmod \, 9),
$$
which implies $B^2_1(N-3\cdot t_2)>0. $
\par
Next, we consider the case when $N\equiv 8 (\bmod \, 9).$ Taking $y=1$ yields
$$
N-3\cdot t_1=N-3\equiv 5 (\bmod \, 9),
$$
which implies $B^2_1(N-3\cdot t_1)>0. $
\end{proof}

\subsubsection{Proof of Theorem \ref{thm:2-triangular} for $b=6$}

\begin{proof}
Multiplying both sides of (\ref{eqn:Cooper;(e)}) by $\psi(q^2)\psi(q^6)$ results in
$$
\psi(q^2)\psi(q^6)a(q)=\varphi(q)\varphi(q^3)\psi(q^2)\psi(q^6)+4q\psi(q^2)^2\psi(q^6)^2.
$$
Using (\ref{eqn:Cooper;(c)}) we obtain
$$
\psi(q^2)\psi(q^6)a(q)=\psi(q)^2 \psi(q^3)^2+4q\psi(q^2)^2 \psi(q^6)^2,
$$
which implies 
\begin{equation}
\label{eqn:relation-tri-1133}
\sum_{n=0}^{\infty} B^{2,6}_1(n) q^n
=
\sum_{n=0}^{\infty} t_{1,1,3,3} (n) q^n+4\sum_{N=0}^{\infty} t_{1,1,3,3} (N) q^{2N+1}.
\end{equation}
Theorem \ref{thm:(1,1,3,3)} implies $t_{1,1,3,3}(n)>0,$ which proves Theorem \ref{thm:2-triangular} for $b=6.$
\end{proof}

\subsection{Proof of Theorem \ref{thm:2-triangular(2)} }

\subsubsection{Preliminary results}

Consider the following result by Dickson \cite[p. 112]{Dickson-2}:

\begin{lemma}
\label{lem-1,16,48}
{\it
A nonnegative integer $n\in\mathbb{N}_0$ can be written as $x^2+16y^2+48z^2$ for $x,y,z\in\mathbb{Z}$ if and only if $n\neq 4l+2, 4l+3, 8l+5, 16l+8, 16l+12, 9^k(9l+6)$, where $k,l\in\mathbb{N}_0.$
}
\end{lemma}

\begin{proposition}
\label{prop-A-1-16}
{\it A nonnegative integer $n\in\mathbb{N}_0$ can be written as $x^2+16(y^2+yz+z^2)$ for $x,y,z\in\mathbb{Z}$ if and only if $n\neq 4l+2, 4l+3, 8l+5, 16l+8, 16l+12, 9^k(9l+6)$, where $k,l\in\mathbb{N}_0.$
}
\end{proposition}

\begin{proof}
Replacing $q$ by $q^{16}$ in (\ref{eqn:Cooper;(f)}) and (\ref{eqn:Cooper;(d)}) results in
\begin{align*}
a(q^{16})=&a(q^{64})+6q^{16}\psi (q^{32})\psi (q^{96}),   \\
\varphi(q^{16})\varphi(q^{48})=&a(q^{64})+2q^{16}\psi (q^{32})\psi (q^{96}).
\end{align*}
Multiplying both sides of these equations by $\varphi(q)$ yields
\begin{align}
\varphi(q)a(q^{16})=&\varphi(q)a(q^{64})+6q^{16}\varphi(q)\psi (q^{32})\psi (q^{96}),   \\
\varphi(q)\varphi(q^{16})\varphi(q^{48})=&\varphi(q)a(q^{64})+2q^{16}\varphi(q)\psi (q^{32})\psi (q^{96}),
\end{align}
which implies
\begin{align*}
\sum_{n=0}^{\infty} A^1_{16}(n) q^n=&\sum_{n=0}^{\infty} A^1_{64}(n) q^n +6q^{16} \sum_{N=0}^{\infty} m_{1\text{-}32,96}(N)q^N,  \\
\sum_{n=0}^{\infty} r_{1,16,48}(n) q^n=&\sum_{n=0}^{\infty} A^1_{64}(n) q^n +2q^{16} \sum_{N=0}^{\infty} m_{1\text{-}32,96}(N)q^N.
\end{align*}
Therefore, it follows that
\begin{align*}
 r_{1,16,48}(n)>0 &\,\mathrm{if \,and \, only \, if}\, A^1_{64}(n)>0 \,\textrm{or} \, m_{1\text{-}32,96}(n-16)>0,  \\
                    &\,\mathrm{if \,and \, only \, if}\, A^1_{16}(n)>0,
\end{align*}
which proves the proposition.
\end{proof}

\begin{proposition}
\label{prop:B-1-2}
{\it A nonnegative integer $N\in\mathbb{N}_0$ can be written as $t_x+2(y^2+yz+z^2)$ for $(x,y,z)\in\mathbb{N}_0\times\mathbb{Z}^2$ if and only if
either of the following occurs:
\begin{enumerate}\itemsep=0pt
\item[$(1)$]
$N\not\equiv 1, 4 \,(\bmod \,9),$
\item[$(2)$]
$N\equiv 1 \,(\bmod \,9)$ and $8N+1=9^k(8 N^{\prime}+1), \,k\in\mathbb{N}, \, N^{\prime}\in\mathbb{N}_0$, and $N^{\prime} \not\equiv 1,4 (\bmod \, 9).$
\end{enumerate}
}
\end{proposition}

\begin{proof}
Multiplying both sides of (\ref{eqn:Cooper;(a)}) by $a(q^{16})$ yields
$$
\varphi(q)a(q^{16})=\varphi(q^4)a(q^{16})+2q\psi(q^8)a(q^{16}),
$$
which implies
$$
\sum_{n=0}^{\infty} A^1_{16}(n) q^n=\sum_{N=0}^{\infty} A^1_4(N) q^{4N}+2\sum_{N=0}^{\infty} B^1_2(N)q^{8N+1}.
$$
We then obtain
$$
B^1_2(N)>0 \,\mathrm{if \,and \, only \, if}\, A^1_{16}(8N+1)>0.
$$
Therefore, the proposition follows from Proposition \ref{prop-A-1-16} and the facts that
$$
8N+1\equiv 0 \,(\bmod \, 9) \,\mathrm{if \,and \, only \, if}\, N\equiv 1 \,(\bmod \, 9),
$$
and
$$
8N+1\equiv 6 \,(\bmod \, 9) \,\mathrm{if \,and \, only \, if}\, N\equiv 4 \,(\bmod \, 9).
$$
\end{proof}

\subsubsection{Proof of Theorem \ref{thm:2-triangular(2)} for $b\neq 3$}

\begin{proof}
From Proposition \ref{prop:B-1-2}, it suffices to show that if $N\equiv 1$ or $4 \,(\bmod \, 9),$ $N\in\mathbb{N}_0$ can be expressed as $t_x+bt_y+2(z^2+zw+w^2)$ for $(x,y,z,w)\in\mathbb{N}_0^2\times \mathbb{Z}^2.$ Assume that $N\equiv 1$ or $4 \,(\bmod \, 9).$
\par
For $b=1,2,$ or $4,$ taking $y=1$ yields
$
N-b\cdot t_1\not\equiv 1,4 (\bmod \, 9),
$
which implies $B^1_2(N-b\cdot t_1)>0.$
\end{proof}

\subsubsection{Proof of Theorem \ref{thm:2-triangular(2)} for $b=3$}

\begin{proof}
Multiplying both sides of (\ref{eqn:Cooper;(f)}) by $\psi(q^2)\psi(q^6)$ results in
$$
\psi(q^2)\psi(q^6)a(q)=\psi(q^2)\psi(q^6)a(q^4)+6q\psi(q^2)^2 \psi(q^6)^2,
$$
which implies 
\begin{equation}
\label{eqn:relation-tri-1133(2)}
\sum_{n=0}^{\infty} B^{2,6}_1 (n) q^n
=
\sum_{N=0}^{\infty} B^{1,3}_2(N) q^{2N}+6\sum_{N=0}^{\infty} t_{1,1,3,3} (N) q^{2N+1}.
\end{equation}
Theorem \ref{thm:2-triangular} implies $B^{2,6}_1 (n)>0,$ which allows us to conclude that $B^{1,3}_2(N)>0.$
\end{proof}

\subsection{Proof of Theorem \ref{thm:2-triangular(3)}}

\subsubsection{Preliminary results}

Consider the following result by Dickson \cite[p. 113]{Dickson-2}:

\begin{lemma}
\label{lem-1,24,72}
{\it A nonnegative integer $n\in\mathbb{N}_0$ can be written as $x^2+24y^2+72z^2$ for $x,y,z\in\mathbb{Z}$ if and only if $n\neq 3l+2, 4l+2, 4l+3, 9l+3, 4^k(8l+5)$, where $k,l\in\mathbb{N}_0.$
}
\end{lemma}

\begin{proposition}
\label{prop-A-1-24}
{\it A nonnegative integer $n\in\mathbb{N}_0$ can be written as $x^2+24(y^2+yz+z^2)$ for $x,y,z\in\mathbb{Z}$ if and only if $n\neq 3l+2, 4l+2, 4l+3, 9l+3, 4^k(8l+5)$, where $k,l\in\mathbb{N}_0.$
}
\end{proposition}

\begin{proof}
Replacing $q$ by $q^{24}$ in (\ref{eqn:Cooper;(f)}) and (\ref{eqn:Cooper;(d)}) implies
\begin{align*}
a(q^{24})=&a(q^{96})+6q^{24}\psi (q^{48})\psi (q^{144}),   \\
\varphi(q^{24})\varphi(q^{72})=&a(q^{96})+2q^{24}\psi (q^{48})\psi (q^{144}).
\end{align*}
Multiplying both sides of these equations by $\varphi(q)$ results in
\begin{align}
\varphi(q)a(q^{24})=&\varphi(q)a(q^{96})+6q^{24}\varphi(q)\psi (q^{48})\psi (q^{144}),   \\
\varphi(q)\varphi(q^{24})\varphi(q^{72})=&\varphi(q)a(q^{96})+2q^{24}\varphi(q)\psi (q^{48})\psi (q^{144}),
\end{align}
which implies
\begin{align*}
\sum_{n=0}^{\infty} A^1_{24}(n) q^n=&\sum_{n=0}^{\infty} A^1_{96}(n) q^n +6q^{24} \sum_{N=0}^{\infty} m_{1\text{-}48,144}(N)q^N,  \\
\sum_{n=0}^{\infty} r_{1,24,72}(n) q^n=&\sum_{n=0}^{\infty} A^1_{96}(n) q^n +2q^{24} \sum_{N=0}^{\infty} m_{1\text{-}48,144}(N)q^N.
\end{align*}
Therefore, it follows that
\begin{align*}
 r_{1,24,72}(n)>0 &\,\mathrm{if \,and \, only \, if}\, A^1_{96}(n)>0 \,\textrm{or} \, m_{1\text{-}48,144}(n-24)>0,  \\
                    &\,\mathrm{if \,and \, only \, if}\, A^1_{24}(n)>0,
\end{align*}
which proves the proposition.
\end{proof}

Proposition \ref{prop-A-1-24} gives rise to the following proposition:

\begin{proposition}
\label{prop:B-1-3}
{\it A nonnegative integer $n\in\mathbb{N}_0$ can be written as $t_x+3(y^2+yz+z^2)$ for $(x,y,z)\in\mathbb{N}_0\times\mathbb{Z}^2$ if and only if
$
n\not\equiv 2,5,7,8 \,(\bmod \, 9).
$
}
\end{proposition}

\begin{proof}
Multiplying both sides of (\ref{eqn:Cooper;(a)}) by $a(q^{24})$ results in
$$
\varphi(q)a(q^{24})=\varphi(q^4)a(q^{24})+2q\psi(q^8)a(q^{24}),
$$
which implies
$$
\sum_{n=0}^{\infty} A^1_{24}(n) q^n=\sum_{N=0}^{\infty} A^1_6(N) q^{4N}+2\sum_{N=0}^{\infty} B^1_3(N)q^{8N+1}.
$$
We then obtain
$$
B^1_3(N)>0 \,\mathrm{if \,and \, only \, if}\, A^1_{24}(8N+1)>0.
$$
Therefore, the proposition follows from Proposition \ref{prop-A-1-24} and the facts that
$$
8N+1\equiv 2 \,(\bmod \, 3) \,\mathrm{if \,and \, only \, if}\, N\equiv 2 \,(\bmod \, 3),
$$
and
$$
8N+1\equiv 3 \,(\bmod \, 9) \,\mathrm{if \,and \, only \, if}\, N\equiv 7 \,(\bmod \, 9).
$$
\end{proof}

\subsubsection{Proof of Theorem \ref{thm:2-triangular(3)} }

\begin{proof}
From Proposition \ref{prop:B-1-3}, it suffices to show that $N\in\mathbb{N}_0$ can be expressed as $t_x+2t_y+3(z^2+zw+w^2)$ for $(x,y,z,w)\in\mathbb{N}_0^2\times\mathbb{Z}^2$ if $N\equiv 2,5,7\, \mathrm{or} \,8 (\bmod \, 9).$
\par
When $N\equiv 2,5 \,\mathrm{or} \,8 \,(\bmod \, 9),$ taking $y=1$ yields
$$
N-2\cdot t_1 \equiv 0,3, \,\mathrm{or} \,6 (\bmod \, 9),
$$
which implies $B^1_3(N-2\cdot t_1)>0.$
\par
When $N\equiv 7 (\bmod \, 9),$ taking $y=2$ yields
$$
N-2\cdot t_2\equiv 1 (\bmod \, 9),
$$
which implies $B^1_3(N-2\cdot t_2)>0.$
\end{proof}

\subsection{Proof of Theorem \ref{thm:2-triangular(4)}}

\subsubsection{Proof of Theorem \ref{thm:2-triangular(4)} for $b=1$}

\begin{proof}
Multiplying both sides of (\ref{eqn:Cooper;(c)}) by $a(q^4)$ results in
\begin{align}
\psi(q)^2a(q^4)=&\varphi(q)\psi(q^2)a(q^4)    \notag \\
                  =&(\varphi(q^4)+2q\psi(q^8)) \psi(q^2)a(q^4), \,\,\,\textrm{by (\ref{eqn:Cooper;(a)}),} \notag \\
                  =&\varphi(q^4)\psi(q^2)a(q^4)+2q\psi(q^2)\psi(q^8)a(q^4),  \notag
\end{align}
which implies
\begin{equation}
\label{eqn:for-even-relation}
\sum_{n=0}^{\infty} B^{1,1}_4(n) q^n=\sum_{N=0}^{\infty} M^{2\textrm{-}1}_2(N)q^{2N}+2\sum_{N=0}^{\infty} B^{1,4}_2(N) q^{2N+1},
\end{equation}
where
$$
M^{2\textrm{-}1}_2(N)=\sharp
\left\{
(x,y,z,w)\in\mathbb{Z}\times\mathbb{N}_0\times\mathbb{Z}^2 \,| \,
N=2x^2+t_y+2(z^2+zw+w^2)
\right\}.
$$
\par
If $n$ is odd and $n=2N+1,$ by Theorem \ref{thm:2-triangular(2)}, 
\begin{align*}
B^{1,1}_4(2N+1)=2B^{1,4}_2(N)>0.
\end{align*}
\par
Consider the case when $n$ is even. By (\ref{eqn:for-even-relation}), 
$$
B^{1,1}_4(2N)=M^{2\textrm{-}1}_2(N).
$$
\par
If $N \not\equiv 1,4 (\bmod \, 9),$ using Proposition \ref{prop:B-1-2}, we obtain $M^{2\textrm{-}1}_2(N)>0.$
If $N \equiv 1 \,\mathrm{or} \, 4 (\bmod \, 9),$ taking $x=1$ yields
$$
N-2x^2\equiv 8 \, \mathrm{or} \,2 (\bmod \, 9) \Longrightarrow B^1_2(N-2\cdot 1^2)>0
                                    \Longrightarrow M^{2\textrm{-}1}_2(N)>0.
$$
\end{proof}

\subsubsection{Proof of Theorem \ref{thm:2-triangular(4)} for $b=2$}

\begin{proof}
Replacing $q$ by $q^4$ in (\ref{eqn:Cooper;(e)}) yields
$$
a(q^4)=\varphi(q^4)\varphi(q^{12})+4q^4\psi(q^8)\psi(q^{24}).
$$
Multiplying both sides of this equation by $\psi(q)\psi(q^2)$ results in
$$
\psi(q)\psi(q^2)a(q^4)=\varphi(q^4)\varphi(q^{12})\psi(q)\psi(q^2)+4q^4\psi(q)\psi(q^2)\psi(q^8)\psi(q^{24}),
$$
which implies 
$$
\sum_{n=0}^{\infty} B^{1,2}_4(n) q^n=
\sum_{n=0}^{\infty} m_{4,12\textrm{-} 1,2} (n)q^n+4q^4\sum_{N=0}^{\infty} t_{1,2,8,24}(N)q^N,
$$
where
$$
m_{4,12\textrm{-} 1,2} (n)=
\sharp
\left\{
(x,y,z,w)\in\mathbb{Z}^2\times\mathbb{N}_0^2 \, | \,
n=4x^2+12y^2+t_z+2t_w
\right\}.
$$
\par
From Guo, Pan, and Sun \cite{Sun}, recall that every $n\in\mathbb{N}_0$ can be expressed as $4x^2+2t_y+t_z$ for $(x,y,z)\in\mathbb{Z}\times\mathbb{N}_0^2,$
which implies 
$$
m_{4,12\textrm{-} 1,2} (n)>0 \Longrightarrow B^{1,2}_4(n)>0.
$$
\end{proof}

\subsection{Proof of Theorem \ref{thm:2-triangular(5)} }

\subsubsection{Preliminary results}

Consider the following result by Dickson \cite[p. 113]{Dickson-2}:

\begin{lemma}
\label{lem-1,40,120}
{\it
A nonnegative integer $n\in\mathbb{N}_0$ can be written as $x^2+40y^2+120z^2$ for $x,y,z\in\mathbb{Z}$ if and only if $n\neq 4l+2, 4l+3, 9^k(9l+6), 25^k(5l\pm2),  4^k(8l+5)$, where $k,l\in\mathbb{N}_0.$
}
\end{lemma}

Using Lemma \ref{lem-1,40,120}, we obtain the following proposition:

\begin{proposition}
\label{prop-A-1-40}
{\it A nonnegative integer $n\in\mathbb{N}_0$ can be written as $x^2+40(y^2+yz+z^2)$ for $x,y,z\in\mathbb{Z}$ if and only if $n\neq 4l+2, 4l+3, 9^k(9l+6), 25^k(5l\pm2),  4^k(8l+5)$, where $k,l\in\mathbb{N}_0.$
}
\end{proposition}

\begin{proof}
Replacing $q$ by $q^{40}$ in (\ref{eqn:Cooper;(f)}) and (\ref{eqn:Cooper;(d)}) results in
\begin{align*}
a(q^{40})=&a(q^{160})+6q^{40}\psi (q^{80})\psi (q^{240}),   \\
\varphi(q^{40})\varphi(q^{120})=&a(q^{160})+2q^{40}\psi (q^{80})\psi (q^{240}).
\end{align*}
Multiplying both sides of these equations by $\varphi(q)$ yields
\begin{align}
\varphi(q)a(q^{40})=&\varphi(q)a(q^{160})+6q^{40}\varphi(q)\psi (q^{80})\psi (q^{240}),   \\
\varphi(q)\varphi(q^{40})\varphi(q^{120})=&\varphi(q)a(q^{160})+2q^{40}\varphi(q)\psi (q^{80})\psi (q^{240}),
\end{align}
which implies
\begin{align*}
\sum_{n=0}^{\infty} A^1_{40}(n) q^n=&\sum_{n=0}^{\infty} A^1_{160}(n) q^n +6q^{40} \sum_{N=0}^{\infty} m_{1\text{-}80,240}(N)q^N,  \\
\sum_{n=0}^{\infty} r_{1,40,120}(n) q^n=&\sum_{n=0}^{\infty} A^1_{160}(n) q^n +2q^{40} \sum_{N=0}^{\infty} m_{1\text{-}80,240}(N)q^N.
\end{align*}
Therefore, it follows that
\begin{align*}
 r_{1,40,120}(n)>0 &\,\mathrm{if \,and \, only \, if}\, A^1_{160}(n)>0 \,\textrm{or} \, m_{1\text{-}80,240}(n-40)>0,  \\
                    &\,\mathrm{if \,and \, only \, if}\, A^1_{40}(n)>0,
\end{align*}
which proves the proposition.
\end{proof}

Proposition \ref{prop-A-1-24} gives rise to th following proposition:

\begin{proposition}
\label{prop:B-1-5}
{\it A nonnegative integer $N\in\mathbb{N}_0$ can be written as $t_x+5(y^2+yz+z^2)$ for $(x,y,z)\in\mathbb{N}_0\times\mathbb{Z}^2$ if and only if
$8N+1\neq 9^k(9l+6), 25^k(5l\pm2)$, where $k,l\in\mathbb{N}_0$; in particular, $N\in\mathbb{N}_0$ can be written as $t_x+5(y^2+yz+z^2)$ for $(x,y,z)\in\mathbb{N}_0\times\mathbb{Z}^2$ if $N\not\equiv 1,4 (\bmod \, 9)$, $N\not\equiv 2,4 (\bmod \, 5),$ and $N\not\equiv 3 (\bmod \, 25).$
}
\end{proposition}

\begin{proof}
Multiplying both sides of (\ref{eqn:Cooper;(a)}) by $a(q^{40})$ yields
$$
\varphi(q)a(q^{40})=\varphi(q^4)a(q^{40})+2q\psi(q^8)a(q^{40}),
$$
which implies
$$
\sum_{n=0}^{\infty} A^1_{40}(n) q^n=\sum_{N=0}^{\infty} A^1_{10}(N) q^{4N}+2\sum_{N=0}^{\infty} B^1_5(N)q^{8N+1}.
$$
We then obtain
$$
B^1_5(N)>0 \,\mathrm{if \,and \, only \, if}\, A^1_{40}(8N+1)>0.
$$
\par
The second statement follows from the facts that
$$
8N+1\equiv 0 \,\mathrm{or} \,6 (\bmod \, 9) \,\mathrm{if \,and \, only \, if}\, N\equiv 1 \mathrm{or} \,4 (\bmod \, 9),
$$
$$
8N+1\equiv \pm2 (\bmod \, 5) \,\mathrm{if \,and \, only \, if}\, N\equiv 2 \,\mathrm{or} \,4 (\bmod \, 5),
$$
and
$$
8N+1\equiv 0 (\bmod \, 25) \,\mathrm{if \,and \, only \, if}\, N \equiv 3 (\bmod \, 25).
$$
\end{proof}

\subsubsection{Proof of Theorem \ref{thm:2-triangular(5)} }

\begin{proof}
By Proposition \ref{prop:B-1-5}, it suffices to show that $n\in\mathbb{N}_0$ can be expressed as $t_x+t_y+5(z^2+zw+w^2)$ for $(x,y,z,w)\in\mathbb{N}_0^2\times \mathbb{Z}^2$ if $n$ satisfies one of the following conditions:

(i) $n\equiv 1 \,\mathrm{or} \,4 (\bmod \, 9),$

(ii) $n\equiv 2 \,\mathrm{or} \,4 (\bmod \, 5),$

(iii) $n\equiv 3 (\bmod \, 25).$

\par
First, suppose that $n\equiv 1 (\bmod \, 9)$ and $n=9N+1$ for $N\in\mathbb{N}_0.$ If $N\not\equiv 1,3 (\bmod \, 5)$ and $N\not\equiv 17 (\bmod \, 25),$
taking $x=1$ yields
$$
n-t_1=9N \equiv 0 (\bmod \, 9), \,\,\not\equiv 2, 4 (\bmod \, 5), \,\,\not\equiv 3 (\bmod \, 25),
$$
which implies $n-t_1$ can be written as $t_y+5(z^2+zw+w^2)$ for $(y,z,w)\in\mathbb{N}_0\times \mathbb{Z}^2.$
\par
If $N\equiv 1 (\bmod \, 5),$ taking $x=4$ yields
$$
n-t_4=9N-9\equiv 0 (\bmod \, 9), \equiv 0 (\bmod \, 5),
$$
which implies $n-t_4$ can be written as $t_y+5(z^2+zw+w^2)$ for $(y,z,w)\in\mathbb{N}_0\times \mathbb{Z}^2.$
\par
If $N\equiv 3 (\bmod \, 5),$ taking $x=2$ results in
$$
n-t_2=9N-2 \equiv 7 (\bmod \, 9), \,\,\equiv 0 (\bmod \, 5),
$$
which implies $n-t_2$ can be written as $t_y+5(z^2+zw+w^2)$ for $(y,z,w)\in\mathbb{N}_0\times \mathbb{Z}^2.$
\par
If $N\equiv 17 (\bmod \, 25),$ taking $x=2$ results in
$$
n-t_2=9N-2\equiv 7 (\bmod \, 9), \,\,\equiv 1 (\bmod \, 5),
$$
which implies $n-t_2$ can be written as $t_y+5(z^2+zw+w^2)$ for $(y,z,w)\in\mathbb{N}_0\times \mathbb{Z}^2.$
\par
Now, suppose $n\equiv 4 (\bmod \, 9)$ and $n=9N+4$ for $N\in\mathbb{N}_0.$ In the same way, we prove that $n$ can be written as $t_x+t_y+5(z^2+zw+w^2)$ for $(x,y,z,w)\in\mathbb{N}_0^2\times \mathbb{Z}^2.$ If $N\not\equiv 1,4(\bmod \, 5)$ and $N\not\equiv 0 (\bmod \, 25),$ we take $x=1.$ If $N\equiv 1 (\bmod \, 5),$ we take $x=7.$
If $N\equiv 4 (\bmod \, 5),$ we take $x=4.$ Finally, if $N\equiv 0 (\bmod \, 25),$ we take $x=3.$
\par
Suppose $n\equiv 2 (\bmod \, 5)$ and $n=5N+2$ for $N\in\mathbb{N}_0.$ If $N\not\equiv 0,6 (\bmod \, 9),$ we take $x=1.$ If $N\equiv 0 \, \mathrm{or} \, 6 (\bmod \, 9),$ we take $x=3.$
\par
Next, assume $n\equiv 4 (\bmod \, 5)$ and $n=5N+4$ for $N\in\mathbb{N}_0.$ If $N\not\equiv 0,6 (\bmod \, 9),$ we take $x=3.$ If $N\equiv 0 \,\mathrm{or} \, 6 (\bmod \, 9),$ we take $x=7.$
\par
Finally, suppose $n\equiv 3 (\bmod \, 25)$ and $n=25N+3$ for $N\in\mathbb{N}_0.$ If $N\not\equiv 4,7 (\bmod \, 9),$ we take $x=2.$ If $N\equiv 4 \,\mathrm{or} \,7 (\bmod \, 9),$ we take $x=4.$
\end{proof}

\section{Proof of Theorem \ref{thm:a,b-c}}

\subsection{Proof of necessary conditions}

\begin{proof}
For fixed positive integers $a$, $b$, and $c$ with $a\le b,$ suppose every $n\in\mathbb{N}$ can be written as $at_x+bt_y+c(z^2+zw+w^2)$ for
$(x,y,z,w)\in\mathbb{N}_0^2\times\mathbb{Z}^2.$
\par
First, assume that $c=1.$ Taking $n=2$ yields $a=1$ or $2.$ If $a=1,$ by Theorem \ref{thm:a-c}, we see that $b$ is arbitrary. If $a=2,$ taking $n=8$ yields $2\le b \le 8.$
\par
Suppose $c=2.$ Taking $n=1$ results in $a=1.$ The choice of $n=4$ implies $1\le b \le 4.$
\par
Next, assume that $c=3.$ Taking $n=1$ results in $a=1.$ Choosing $n=2$ implies $b=1$ or $2.$ Taking $n=8$ implies $b=2.$
\par
Suppose $c=4.$ Taking $n=1,$ we have $a=1.$ Choosing $n=2$ implies $b=1$ or $2.$
\par
Assume $c=5.$ Taking $n=1$ yields $a=1.$ Choosing $n=2$ implies $b=1$ or $2$, and taking $n=4$ implies $b=1.$
\par
Finally, suppose $c\ge 6.$ Taking $n=1$ results in $a=1.$ Choosing $n=2$ implies $b=1$ or $2$; taking $n=4$ implies $b=1.$ On the other hand, $n=5$ cannot be expressed as $t_x+t_y+c(z^2+zw+w^2)$ for $(x,y,z,w)\in\mathbb{N}_0^2\times\mathbb{Z}^2,$ which is a contradiction.
\end{proof}

\subsection{Proof of sufficient conditions}

\begin{proof}
Sufficiency follows from Theorems \ref{thm:2-triangular}, \ref{thm:2-triangular(2)}, \ref{thm:2-triangular(3)}, \ref{thm:2-triangular(4)}, and
\ref{thm:2-triangular(5)}. Note that Theorem \ref{thm:a,b-c} (2) follows from the proof of the necessary conditions.
\end{proof}

\subsubsection*{Acknowledgment}
We are grateful for Professor K. S. Williams and Professor B. Landman for their useful comments. We also would like to thank the referee for recommending various improvements.

\end{document}